\documentclass[onefignum]{siamart190516}

\usepackage{amsfonts}
\usepackage{graphicx}
\usepackage{epstopdf}
\usepackage{algorithmic}
\usepackage{amssymb}

\ifpdf
  \DeclareGraphicsExtensions{.eps,.pdf,.png,.jpg}
\else
  \DeclareGraphicsExtensions{.eps}
\fi

\newsiamremark{remark}{Remark}
\newsiamremark{hypothesis}{Hypothesis}
\crefname{hypothesis}{Hypothesis}{Hypotheses}
\newsiamthm{claim}{Claim}
\newsiamthm{example}{Example}
\newsiamthm{conjecture}{Conjecture}
\headers{Convergence rate analysis of  the gradient
descent-ascent method}{ M. Zamani, H. Abbaszadehpeivasti, and E. de Klerk}

\title{Convergence rate analysis of  the gradient
descent-ascent method for convex-concave saddle-point problems\thanks{This work was supported by the Dutch Scientific Council (NWO) grant OCENW.GROOT.2019.015, Optimization for and with Machine Learning (OPTIMAL).}}

\author{Moslem Zamani \thanks{Tilburg University, Department of Econometrics and Operations Research, Tilburg, The Netherlands (\email{m.zamani\_1@tilburguniversity.edu, h.abbaszadehpeivasti@tilburguniversity.edu, e.deklerk@tilburguniversity.edu}).}
\and Hadi Abbaszadehpeivasti\footnotemark[2]
\and Etienne de Klerk \footnotemark[2]}

\usepackage{amsopn}

\usepackage{todonotes}
\usepackage{enumerate}

\DeclareMathOperator{\st}{s.t.}

\newcommand{\vertiii}[1]{{\left\vert\kern-0.25ex\left\vert\kern-0.25ex\left\vert #1
    \right\vert\kern-0.25ex\right\vert\kern-0.25ex\right\vert}}

\ifpdf
\hypersetup{
  pdftitle={Convergence rate analysis of  the gradient
descent-ascent method for convex-concave saddle-point problems},
  pdfauthor={Moslem Zamani, Hadi Abbaszadehpeivasti, and Etienne de Klerk}
}
\fi

\begin{document}

\maketitle

\begin{abstract}
In this paper, we study the  gradient
descent-ascent  method for convex-concave saddle-point problems. We derive a new non-asymptotic global convergence rate in terms of distance to the solution set  by using the semidefinite programming performance estimation method. The given convergence rate incorporates most parameters of the problem and it is exact for a large class of strongly convex-strongly concave saddle-point problems for one iteration.  We also investigate the algorithm without strong convexity and we provide some necessary and sufficient conditions under which the gradient descent-ascent enjoys linear convergence.
\end{abstract}


\section{Introduction}\label{intro}

We consider the convex-concave saddle point problem
\begin{align}\label{P}
\min_{x\in \mathbb{R}^n}\max_{y\in \mathbb{R}^m} F(x, y),
\end{align}
where $F:\mathbb{R}^n\times\mathbb{R}^m\to (-\infty, \infty)$, and $F(\cdot, y)$ and $F(x, \cdot)$ are convex and  concave, respectively, for any fixed $x\in\mathbb{R}^n$ and $y\in\mathbb{R}^m$.  We assume that problem \eqref{P} has some solution, that is, there exists $(x^\star, y^\star)\in\mathbb{R}^n\times\mathbb{R}^m$ with
$$
F(x^\star, y)\leq F(x^\star, y^\star)\leq F(x, y^\star), \ \ \forall x\in\mathbb{R}^n, \forall y\in\mathbb{R}^m.
$$
We denote the solution set of problem \eqref{P} with $S^\star$. We call $F$ smooth if for some $L_x, L_y, L_{xy}$, we have
\begin{align*}
& i) &&\|\nabla_x F(x_2, y)-\nabla_x F(x_1, y)\|\leq L_x\|x_2-x_1\| &&& \forall x_1, x_2, y \\
& ii)  &&\|\nabla_y F(x, y_2)-\nabla_y F(x, y_1)\|\leq L_y\|y_2-y_1\| &&& \forall x, y_1, y_2 \\
& iii) &&\|\nabla_x F(x, y_2)-\nabla_x F(x, y_1)\|\leq L_{xy}\|y_2-y_1\| &&& \forall x, y_1, y_2 \\
& iv) &&\|\nabla_y F(x_2, y)-\nabla_y F(x_1, y)\|\leq L_{xy}\|x_2-x_1\| &&& \forall x_1, x_2, y.
\end{align*}
The function $F$ is said to be strongly convex-strongly concave if
\begin{align*}
& i) \  F(\cdot, y)-\tfrac{\mu_x}{2}\|\cdot\|^2 \textrm{\ is convex for any fixed}\ y\\
& ii) \ F(x,\cdot)+\tfrac{\mu_y}{2}\|\cdot\|^2 \textrm{\ is concave for any fixed}\ x,
\end{align*}
for some $\mu_x, \mu_y>0$.
Note that strong convex-strong concavity implies that problem \eqref{P} has a unique solution $(x^\star, y^\star)$. We denote the set of smooth strongly convex-strongly concave functions by $\mathcal{F}(L_x, L_y, L_{xy}, \mu_x, \mu_y)$.

Problem \eqref{P} has applications in game theory \cite{bacsar1998dynamic}, robust optimization \cite{ben2009robust}, adversarial training \cite{goodfellow2014generative}, and reinforcement learning \cite{yang2020off}, to name but a few. Moreover, various algorithms have been developed for solving saddle point problems; see e.g. \cite{hamedani2021primal, jiang2022generalized, lin2020near, nie2022saddle, simpson2019input, wang2020improved}.

One of the simplest approaches for handling problem \eqref{P} introduced in \cite[Chapter 6]{arrow1958studies} is the gradient-descent-ascent method,  which may be regarded as a generalization of the gradient method for saddle point problems. The gradient descent-ascent method is described in Algorithm \ref{GDA}.

\begin{algorithm}
\caption{The gradient descent-ascent method}
\label{GDA}
\begin{algorithmic}
\STATE Set $N$ and $t>0$ (step length), pick $x^1$ and $y^1$.
\STATE For $k=1, 2, \ldots, N$ perform the following steps:\\
\begin{enumerate}
\item
$x^{k+1}=x^k-t\nabla_x F(x^k, y^k)$.
\item
$y^{k+1}=y^k+t\nabla_y F(x^k, y^k)$.
\end{enumerate}
\end{algorithmic}
\end{algorithm}

The local and global linear convergence of Algorithm \ref{GDA} have been investigated in the literature; see \cite{fallah2020optimal, liang2019interaction,  zhang2022near} and the references therein. As we investigate the global linear convergence rate of Algorithm \ref{GDA}, we mention one known global convergence result, which is derived by using variational inequality techniques. Suppose that $z=(x,y)$. Let the function $\phi:\mathbb{R}^{n+m}\to \mathbb{R}^{n+m}$ given by
$\phi(z)=\begin{pmatrix}
\nabla_x F(z) & -\nabla_y F(z)
\end{pmatrix}^T$. It is shown that, see e.g. \cite{mokhtari2020convergence},
\begin{align*}
 & \|\phi(\bar z)-\phi(\hat z)\|\leq 2L\|\bar z-\hat z\|,  \\
 & \langle \phi(\bar z)- \phi(\hat z), \bar z- \hat z\rangle \geq \mu\|\bar z-\hat z\|^2,
\end{align*}
where $L=\max\{L_x, L_y, L_{xy}\}$ and $\mu=\min\{\mu_x, \mu_y\}$. Indeed, $\phi$ is Lipschitz continuous and strongly monotone. By \cite[Theorem 12.1.2]{facchinei2003finite}, for $t\in(0, \tfrac{\mu}{2L^2})$, we have
\begin{align}\label{V.GDA}
\|x^2-x^\star\|^2+\|y^2-y^\star\|^2\leq
(1+4L^2t^2-2\mu t)\left(\|x^1-x^\star\|^2+\|y^1-y^\star\|^2\right).
\end{align}

In this study, we revisit Algorithm \ref{GDA} and improve the convergence rate \eqref{V.GDA}. Indeed, we  derive a new convergence rate involving most parameters of problem \eqref{P}. It is worth noting that if one sets $L=\max\{L_x, L_y, L_{xy}\}$ and $\mu=\min\{\mu_x, \mu_y\}$, the new bound dominates the convergence rate \eqref{V.GDA} for any step length $t\in\left(0, \tfrac{\mu}{2L^2}\right)$. Furthermore, by setting $t=\tfrac{\mu}{4L^2}$, one can infer that Algorithm \ref{GDA} has a complexity of
$\mathcal{O}\left(\tfrac{L^2}{\mu^2}\ln \left(\tfrac{1}{\epsilon} \right)\right)$, which is the known iteration complexity bound in the literature; see e.g.  \cite{beznosikov2022smooth, zhang2021unified}. In this study, thanks to the new convergence rate given in Theorem \ref{T_GDA},
the order of complexity of $\mathcal{O}\left(\left(\tfrac{L}{\mu}+\tfrac{L_{xy}^2}{\mu^2}\right)\ln \left(\tfrac{1}{\epsilon} \right)\right)$ is obtained
when $L=\max\{L_x, L_y\}$ and $\mu=\min\{\mu_x, \mu_y\}$, which is more informative in comparison with the above-mentioned one. Moreover, by providing some example, we show that the given convergence rate is exact for one iteration.

The paper is organized as follows. First, we present basic definitions and preliminaries used to establish the results. Section \ref{Sec.1} is devoted to the study of the linear convergence of Algorithm \ref{GDA}. In Section \ref{Sec.2}, we study the linear convergence of the gradient descent-ascent method without strong convexity. Indeed, we let $F\in\mathcal{F}(L_x, L_y, L_{xy}, 0, 0)$ and give some necessary and sufficient conditions for the linear convergence. Moreover, we derive a convergence rate under this setting.

\subsubsection*{Notation}
The $n$-dimensional Euclidean space is denoted by $\mathbb{R}^n$.
 We use $\langle \cdot, \cdot\rangle$ and $\| \cdot\|$ to denote the Euclidean inner product and norm, respectively.
  For a matrix $A$, $A_{ij}$ denotes its $(i, j)$-th entry,
  and $A^T$ represents the transpose of $A$.  We use $\lambda_{\max} (A)$ and $\lambda_{\min} (A)$ to denote  the largest and the smallest eigenvalue of symmetric matrix $A$, respectively.

Let $X\subseteq \mathbb{R}^n$. We denote the distance function to $X$ by $d_X(x):=\inf_{\bar x\in X} \|x-\bar x\|$ and the set-valued mapping $\Pi_{X} (x)$ stands for the projection of $x$ on $X$, i.e., $\Pi_{X} (x):=\{y\in X: \|x-y\|=d_X(x)\}$.

 We call a differentiable function $f:\mathbb{R}^n\to(-\infty, \infty)$ $L$-smooth if
$$
\|\nabla f(x_1)-\nabla f(x_2)\|\leq L\|x_1-x_2\| \ \ \forall  x_1, x_2\in\mathbb{R}^n.
$$
The function $f:\mathbb{R}^n\to\mathbb{R}$ is called $\mu$-strongly convex function if the function $x \mapsto f(x)-\tfrac{\mu}{2}\| x\|^2$ is convex.
 Clearly, any convex function is $0$-strongly convex.
  We denote the set of real-valued convex functions which are $L$-smooth and $\mu$-strongly convex by $\mathcal{F}_{\mu,L}(\mathbb{R}^n)$.

Let $\mathcal{I}$ be a finite index set and let $\left\{(x^i; g^i; f^i)\right\}_{i\in \mathcal{I}}\subseteq \mathbb{R}^n\times \mathbb{R}^n\times \mathbb{R}$.
A set $\left\{(x^i; g^i; f^i)\right\}_{i\in \mathcal{I}}$ is called $\mathcal{F}_{\mu,L}$-interpolable if there exists $f\in\mathcal{F}_{\mu,L}(\mathbb{R}^n)$
 with
$$
f(x^i)=f^i, \ g^i\in\partial f(x^i) \ \ i\in\mathcal{I}.
$$
The next theorem gives necessary and sufficient conditions for $\mathcal{F}_{\mu,L}$-interpolablity.

\begin{theorem}\cite[Theorem 4]{taylor2017smooth}\label{T1}
Let $L\in (0, \infty)$ and $\mu\in [0, \infty)$ and let $\mathcal{I}$ be a finite index set. The set
$\left\{(x^i; g^i; f^i)\right\}_{i\in \mathcal{I}}\subseteq \mathbb{R}^n\times \mathbb{R}^n \times \mathbb{R}$ is $\mathcal{F}_{\mu,L}$-interpolable if and only if for any $i, j\in\mathcal{I}$, we have
{\small{
\begin{align}\label{interp.1}
\tfrac{1}{2(1-\tfrac{\mu}{L})}\left(\tfrac{1}{L}\left\|g^i-g^j\right\|^2+\mu\left\|x^i-x^j\right\|^2-\tfrac{2\mu}{L}\left\langle g^j-g^i,x^j-x^i\right\rangle\right)\leq f^i-f^j-\left\langle g^j, x^i-x^j\right\rangle.
\end{align}
}}
\end{theorem}

It is worth mentioning that, under the assumptions of Theorem \ref{T1}, the set $\left\{(x^i; g^i; f^i)\right\}_{i\in \mathcal{I}}$ is interpolable with an $L$-smooth $\mu$-strongly concave function if and only if for any $i, j\in\mathcal{I}$, we have

{\small{
\begin{align}\label{interp.2}
\tfrac{1}{2(1-\tfrac{\mu}{L})}\left(\tfrac{1}{L}\left\|g^i-g^j\right\|^2+\mu\left\|x^i-x^j\right\|^2+\tfrac{2\mu}{L}\left\langle g^j-g^i,x^j-x^i\right\rangle\right)\leq -f^i+f^j+\left\langle g^j, x^i-x^j\right\rangle.
\end{align}
}}

\section{The gradient descent-ascent method}\label{Sec.1}

In this section, we study  the convergence rate of gradient descent-ascent method when
 $F\in\mathcal{F}(L_x, L_y, L_{xy}, \mu_x, \mu_y)$ with $\min\{\mu_x, \mu_y\}>0$.
Indeed, we investigate the worst-case behavior of one step of  Algorithm \ref{GDA} in terms of distance to the unique saddle point $(x^\star, y^\star)$. The worst-cast convergence rate of Algorithm \ref{GDA} may be computed by solving the following abstract optimization problem,
\begin{align}\label{P1}
\nonumber   \max & \ \frac{\|x^2-x^\star\|^2+\|y^2-y^\star\|^2}{\|x^1-x^\star\|^2+\|y^1-y^\star\|^2}\\
\st   &  \  (x^2, y^2) \ \textrm{is generated  by Algorithm \ref{GDA} w.r.t.}\ F, x^1, y^1  \\
\nonumber  & \ (x^\star, y^\star) \textrm{\ is the unique saddle point of problem \eqref{P}} \\
\nonumber  & \ F\in\mathcal{F}(L_x, L_y, L_{xy}, \mu_x, \mu_y) \\
\nonumber  & \ x^1\in\mathbb{R}^n, y^1\in\mathbb{R}^m.
\end{align}
In problem \eqref{P1}, $F, x^1, x^2, x^\star, y^1, y^2, y^\star$ are decision variables and $\mu_x, L_x, \mu_y, L_y, L_{xy}, t$ are fixed parameters. To handle problem \eqref{P1}, we employ performance estimation method introduced in \cite{drori2014performance}.

Suppose that
\begin{align*}
 & F^{i,j}=F(x^i, y^j) \ \ \ \ \ \ \ \ && i,j\in\{1, 2, \star\}, \\
 & G_x^{i,j}=\nabla_x F(x^i, y^j) \ \ \ \ \ \ \ \ && i,j\in\{1, 2, \star\}, \\
 & G_y^{i,j}=\nabla_y F(x^i, y^j)  \ \ \ \ \ \ \ \ && i,j\in\{1, 2, \star\}. \\
\end{align*}
Note that due to the  the necessary and sufficient conditions for convex-concave saddle point problems, we have
$$
G_x^{\star, \star}=0, \ \ \ \ \ G_y^{\star, \star}=0.
$$
By using Theorem \ref{T1}, problem \eqref{P1} may be relaxed as a finite dimensional optimization problem,
\begin{align}\label{P2}
\nonumber   \max & \ \frac{\|x^2-x^\star\|^2+\|y^2-y^\star\|^2}{\|x^1-x^\star\|^2+\|y^1-y^\star\|^2}\\
 \nonumber \st   &  \  \{(x^1; G_x^{1,k}; F^{1,k}),(x^2; G_x^{2,k}; F^{2,k}),(x^\star;  G_x^{\star,k}; F^{\star,k})\}  \ \textrm{satisty \eqref{interp.1} for}\\
  \nonumber    &\ \ \ \ \  \textrm{ $k\in\{1,2, \star\}$ w.r.t.}\ \mu_x, L_x \\
   \nonumber    &  \  \{(y^1; G_y^{k,1}; F^{k,1}),(y^2; G_y^{k,2}; F^{k,2}),(y^\star;  G_y^{k,\star}; F^{k,\star})\}  \ \textrm{satisty \eqref{interp.2} for}\\
     &\ \ \ \ \  \textrm{ $k\in\{1,2, \star\}$ w.r.t.}\ \mu_y, L_y \\
  \nonumber & \|G_x^{k,i}-G_x^{k,j}\|\leq L_{xy}\|y^i-y^j\|, \ \ \ i,j,k\in\{1, 2, \star\}\\
   \nonumber & \|G_y^{i,k}-G_y^{j,k}\|\leq L_{xy}\|x^i-x^j\|, \ \ \ i,j,k\in\{1, 2, \star\}\\
  \nonumber   & \ x^2=x^1-tG_x^{1,1}\\
  \nonumber & \ y^2=y^1+tG_y^{1,1},\\
\nonumber &   G_x^{\star, \star}=0, \ G_y^{\star, \star}=0.
\end{align}
In problem \eqref{P2}, $\{(x^i; G_x^{i,j}; F^{i,j})\}$ and $\{(y^i; G_y^{j,i}; F^{j,i})\}$ ($i,j\in\{1, 2, \star\}$) are decision variables. We may assume that $x^\star=0$ and $y^\star=0$ as Algorithm \ref{GDA} is invariant under translation.  By elimination, problem \eqref{P2} may be reformulated as follows,
\begin{align}\label{P3}
\nonumber   \max & \ \frac{\|x^1-tG_x^{1,1}\|^2+\|y^1+tG_y^{1,1}\|^2}{\|x^1\|^2+\|y^1\|^2}\\
 \nonumber \st   &  \  \{(x^1; G_x^{1,k}; F^{1,k}),(x^1-tG_x^{1,1}; G_x^{2,k}; F^{2,k}),(0;  G_x^{\star,k}; F^{\star,k})\}  \ \textrm{satisty \eqref{interp.1} for}\\
  \nonumber    &\ \ \ \ \  \textrm{ $k\in\{1,2, \star\}$ w.r.t.}\ \mu_x, L_x \\
   \nonumber    &  \  \{(y^1; G_y^{k,1}; F^{k,1}),(y^1+tG_y^{1,1}; G_y^{k,2}; F^{k,2}),(0;  G_y^{k,\star}; F^{k,\star})\}  \ \textrm{satisty \eqref{interp.2} for}\\
     &\ \ \ \ \  \textrm{ $k\in\{1,2, \star\}$ w.r.t.}\ \mu_y, L_y \\
  \nonumber & \|G_x^{k,i}-G_x^{k,j}\|\leq L_{xy}\|y^i-y^j\|, \ \ \ i,j,k\in\{1, 2, \star\}\\
   \nonumber & \|G_y^{i,k}-G_y^{j,k}\|\leq L_{xy}\|x^i-x^j\|, \ \ \ i,j,k\in\{1, 2, \star\}\\
\nonumber &   G_x^{\star, \star}=0, \ G_y^{\star, \star}=0.
\end{align}

To handle problem \eqref{P3}, we formulate a semi-definite program by using the Gram matrix of the unknown vectors in the problem. Indeed, we form the Gram matrices $X$ and $Y$ corresponding to $\{(x^i; G_x^{i,j})\}$ and $\{(y^i; G_y^{j,i})\}$ ($i,j\in\{1, 2, \star\}$), respectively. The interested reader can refer to \cite{taylor2017smooth, zamani2022exact} for more details concerning the Gram matrix formulation.

 For the convenience of the analysis, we investigate the linear convergence of Algorithm \ref{GDA} in terms of  $L=\max\{L_x, L_y\}$ and  $\mu=\min\{\mu_x, \mu_y\}$.
Before we present the main theorem in this section, we need to present a lemma.
\begin{lemma}\label{Lemma1}
Let $0<\mu\leq L$, $c\geq 0$ and let $I=\left(0, \tfrac{2\mu}{\mu L+c^2}\right)$. Suppose that the function $u:I\to\mathbb{R}$ given by
$$
u(t)=\tfrac{1}{2}\left(L^2+\mu^2+2c^2\right)t^2-(L+\mu)t +\tfrac{1}{2}(L-\mu)t\sqrt{(Lt + \mu t - 2)^2 + 4c^2t^2}.
$$
Then $u$ is convex on $I$ and $u(I)\subseteq [-1, 0)$.
\end{lemma}
\begin{proof}
Consider the function $v:I\to\mathbb{R}$ given by
$$
v(t)=\left(L^2+\mu^2+2c^2\right)t+(L-\mu)\sqrt{(Lt + \mu t - 2)^2 + 4c^2t^2}.
$$
The function $v$ is convex and positive on $I$. By elementary calculus, one can show that $v^\prime(0)>0$. So $v$ is increasing on $I$ due to the convexity. As the product of positive monotone convex functions is a convex function, the function $t \mapsto tv(t)$ is also convex, which implies the convexity of $u$. Indeed, $u$ is strictly convex on $I$. Since strictly convex functions attain their maximum on endpoints of a given interval, $u(t)<\max\{u(0), u(\tfrac{2\mu}{\mu L+c^2})\}=0$ for $t\in I$. It remains to show that $\min_{t\in I} u(t)\geq -1$. This follows from the point that
$$
u(t)\geq \tfrac{1}{2}\left(L^2+\mu^2\right)t^2-(L+\mu)t\geq \tfrac{-1}{2}\left(1+\tfrac{2L\mu}{L^2+\mu^2}\right)\geq -1,
$$
and the proof is complete.
\end{proof}

In the next theorem, we get an upper bound for problem \eqref{P3} by using weak duality.

\begin{theorem}\label{T_GDA}
Let $F\in\mathcal{F}(L_x, L_y, L_{xy}, \mu_x, \mu_y)$. Suppose that $L=\max\{L_x, L_y\}$ and $\mu=\min\{\mu_x, \mu_y\}>0$. If  $t\in\left(0, \tfrac{2\mu}{\mu L+L_{xy}^2}\right)$, then Algorithm \ref{GDA} generates $(x^2, y^2)$ such that
\begin{align}\label{C_GDA}
\|x^2-x^\star\|^2+\|y^2-y^\star\|^2\leq
\alpha\left(\|x^1-x^\star\|^2+\|y^1-y^\star\|^2\right),
\end{align}
where
$$
\alpha=1+\tfrac{1}{2}\left(L^2+\mu^2+2L_{xy}^2\right)t^2-(L+\mu)t +\tfrac{1}{2}(L-\mu)t\sqrt{(Lt + \mu t - 2)^2 + 4L_{xy}^2t^2}.
$$
\end{theorem}
\begin{proof}
As mentioned earlier, we assume without loss of generality that $x^\star=0$ and $y^\star=0$. By the assumptions, $F(\cdot, y)\in\mathcal{F}_{\mu,L}(\mathbb{R}^n)$ and $F(x, \cdot)\in\mathcal{F}_{\mu,L}(\mathbb{R}^m)$ for any fixed $x, y$.  Without loss of generality, we may assume that $L_{xy}=1$. This follows from the point that Algorithm \ref{GDA} under the given assumptions generate the same point $(x^2, y^2)$ for the problem
\begin{align*}
\min_{x\in \mathbb{R}^n}\max_{y\in \mathbb{R}^m} \tfrac{1}{L_{xy}}F(x, y),
\end{align*}
with the step length $L_{xy}t$.
Let $t\in\left(0, \tfrac{2\mu}{\mu L+1}\right)$ and
\begin{align*}
 &\bar\alpha=1+\tfrac{1}{2}\left(L^2+\mu^2+2\right)t^2-(L+\mu)t +\tfrac{1}{2}(L-\mu)t\sqrt{(Lt + \mu t - 2)^2 + 4t^2},\\
  & \beta=\sqrt{(Lt + \mu t - 2)^2 + 4t^2}, \ \ \ \ \ \ \ \ \ \ \ \ \ \ \ \ \
   \gamma_1=\tfrac{t \left(-\beta+t \left(L \left(\beta+t (L+\mu)-3\right)-\mu+2 t\right)+2\right)}{\beta},\\
   &\gamma_2=\tfrac{t \left(t^2 (\mu (L+\mu)+2)-\mu t \beta+\beta-t (L+3 \mu)+2\right)}{\beta}, \ \ \ \gamma_3=\tfrac{t^2 \left(\beta+L t-\mu t\right)}{2 \beta}.
\end{align*}
It is  easily verified that $ \gamma_1, \gamma_2, \gamma_3\geq 0$. Moreover, Lemma \ref{Lemma1} implies that $\bar\alpha\in [0, 1)$. By doing some algebra, one can show that
{\small{
\begin{align*}
& \left\|x^1-tG_x^{1,1}\right\|^2+\left\|y^1+tG_y^{1,1}\right\|^2-
\bar\alpha\left(\left\|x^1\right\|^2+\left\|y^1\right\|^2\right)+\gamma_1\bigg(F^{1,1}-F^{\star,1}-\left\langle G_x^{\star,1},x^1\right\rangle-\\
  &\tfrac{L}{2\left(L-\mu\right)}\Big(\tfrac{1}{L}\left\|G_x^{1,1}-G_x^{\star,1}\right\|^2+\mu\left\|x^1\right\|^2-\tfrac{2\mu}{L}\left\langle G_x^{\star,1}-G_x^{1,1},-x^1\right\rangle\Big)\bigg)+\gamma_2\bigg(F^{\star,1}-F^{1,1}+\\
  &\left\langle G_x^{1,1},x^1\right\rangle-\tfrac{L}{2\left(L-\mu\right)}\Big(\tfrac{1}{L}\left\|G_x^{\star,1}-G_x^{1,1}\right\|^2+\mu\left\|x^1\right\|^2-\tfrac{2\mu}{L}\left\langle G_x^{1,1}-G_x^{\star,1},x^1\right\rangle\Big)\bigg)+\gamma_2\bigg(F^{1,\star}-\\
  & F^{\star,\star}-\tfrac{L}{2\left(L-\mu\right)}\Big(\tfrac{1}{L}\left\|G_x^{1,\star}\right\|^2+\mu\left\|x^1\right\|^2-\tfrac{2\mu}{L}\left\langle G_x^{1,\star},x^1\right\rangle\Big)\bigg)+\gamma_1\bigg(F^{\star,\star}-F^{1,\star}+\left\langle G_x^{1,\star},x^1\right\rangle-\\
  & \tfrac{L}{2\left(L-\mu\right)}\Big(\tfrac{1}{L}\left\|G_x^{1,\star}\right\|^2+\mu\left\|x^1\right\|^2-\tfrac{2\mu}{L}\left\langle G_x^{1,\star},x^1\right\rangle\Big)\bigg)+\gamma_1\bigg(F^{1,\star}-F^{1,1}+\left\langle G_y^{1,\star},y^1\right\rangle-\tfrac{L}{2\left(L-\mu\right)}\\
  & \Big(\tfrac{1}{L}\left\|G_y^{1,1}-G_y^{1,\star}\right\|^2+\mu\left\|y^1\right\|^2-\tfrac{2\mu}{L}\left\langle G_y^{1,\star}-G_y^{1,1},y^1\right\rangle\Big)\bigg)+\gamma_2\bigg(F^{1,1}-F^{1,\star}-\left\langle G_y^{1,1},y^1\right\rangle-\\
  &\tfrac{L}{2\left(L-\mu\right)}\Big(\tfrac{1}{L}\left\|G_y^{1,\star}-G_y^{1,1}\right\|^2+\mu\left\|y^1\right\|^2-\tfrac{2\mu}{L}\left\langle-G_y^{1,1}+G_y^{1,\star},y^1\right\rangle\Big)\bigg)+\gamma_2\bigg(-F^{\star,1}+F^{\star,\star}-\\
  & \tfrac{L}{2\left(L-\mu\right)}\Big(\tfrac{1}{L}\left\|G_y^{\star,1}\right\|^2+\mu\left\|y^1\right\|^2-\tfrac{2\mu}{L}\left\langle G_y^{\star,1},-y^1\right\rangle\Big)\bigg)+\gamma_1\bigg(-F^{\star,\star}+F^{\star,1}+\left\langle G_y^{\star,1},-y^1\right\rangle-\\
  &\tfrac{L}{2\left(L-\mu\right)}\Big(\tfrac{1}{L}\left\|G_y^{\star,1}\right\|^2+\mu\left\|y^1\right\|^2-\tfrac{2\mu}{L}\left\langle-G_y^{\star,1},y^1\right\rangle\Big)\bigg)+\gamma_3\left(\left\|x^1\right\|^2-\left\|G_y^{1,1}-G_y^{\star,1}\right\|^2\right)+\\
  & \gamma_3\left(\left\|x^1\right\|^2-\left\|G_y^{1,\star}\right\|^2\right)+\gamma_3\left(\left\|y^1\right\|^2-\left\|G_x^{1,1}-G_x^{1,\star}\right\|^2\right)+\gamma_3\left(\left\|y^1\right\|^2-\left\|G_x^{\star,1}\right\|^2\right)\\
  = & -\zeta_1\left\|x^1-\zeta_2 G_x^{1,1}-\zeta_3(G_x^{1,\star}-G_x^{\star,1})\right\|^2-\zeta_4\left\|G_x^{1,1}-G_x^{1,\star}-G_x^{\star,1}\right\|^2-\\
  & \zeta_1\left\|y^1+\zeta_2 G_y^{1,1}-\zeta_3(G_y^{1,\star}-G_y^{\star,1})\right\|^2-\zeta_4\left\|G_y^{1,1}-G_y^{\star,1}-G_y^{1,\star}\right\|^2\leq 0,
\end{align*}
}}
where  $\zeta_1, \zeta_2, \zeta_3, \zeta_4$ are
\begin{align*}
  & \zeta_1=\tfrac{1}{2} t \left(\tfrac{\left(L^2+\mu^2\right) \beta}{L-\mu}-\tfrac{2 t^2 (L-\mu)}{\beta}+(L+\mu) (t (L+\mu)-2)\right), \\
  & \zeta_2=-\tfrac{\left(L^2 t-L-\mu^2 t+\mu\right) \beta-L^2 t (L t+\mu t-3)-(L+\mu) \left(\mu^2 t^2-2 \mu t+2 t^2+2\right)+\mu^2 t}{2 t^2 (L+\mu)^2 (L \mu+1)-8 L \mu t (L+\mu)+8 L \mu}, \\
  & \zeta_3=-\tfrac{t \left(L^2+6 L \mu+\mu^2\right)-2 t^2 (L+\mu) (L \mu+1)-(L-\mu) \beta-2 (L+\mu)}{2 t^2 (L+\mu)^2 (L \mu+1)-8 L \mu t (L+\mu)+8 L \mu}, \\
  & \zeta_4=\tfrac{t \left(\beta+L t-\mu t\right)^2}{4 (L-\mu) \beta}.
\end{align*}
Note that $\zeta_1, \zeta_4\geq 0$. Therefore, any feasible solution of problem \eqref{P3} satisfies
\begin{align*}
   \frac{\left\|x^1-tG_x^{1,1}\right\|^2+\left\|y^1+tG_y^{1,1}\right\|^2}{\left\|x^1\right\|^2+\left\|y^1\right\|^2}\leq \bar\alpha.
\end{align*}
The proof will be complete by a suitable scaling.
\end{proof}

One may wonder how we obtained the convergence rate  in Theorem \ref{T_GDA}.
Consider the optimization problem
\begin{align}\label{QS2}
\min_{x\in \mathbb{R}^n} f(x),
\end{align}
where $f\in\mathcal{F}_{\mu, L}$. It is known that the quadratic function $q(x)=x^TQx$ with $\lambda_{\max}(Q)=L$ and $\lambda_{\min}(Q)=\mu$ attains  the worst-case convergence rate for the gradient method; see e.g. \cite{de2017worst}. We guessed that this property may hold for problem \eqref{P} and we investigated the bilinear saddle point problem
\begin{align}\label{bP}
\min_{x\in \mathbb{R}^2}\max_{y\in \mathbb{R}^2} \tfrac{1}{2}x^T\begin{pmatrix}
L_x & 0\\
0&\mu_x
\end{pmatrix}x+x^T\begin{pmatrix}
    0 & L_{xy} \\
    L_{xy} & 0
  \end{pmatrix}y-\tfrac{1}{2}y^T\begin{pmatrix}
L_y & 0\\
0&\mu_y
\end{pmatrix}y,
\end{align}
where $L_x\geq \mu_x>0$, $L_y\geq \mu_y>0$ and $L_{xy}$ are fixed parameters and we derived the worst case convergence of Algorithm \ref{GDA} with respect to this problem. Our numerical experiments showed that the derived convergence rate is the same as the optimal value of the semi-definite programming problem corresponding to problem \eqref{P3}. Moreover, as a by-product, we exhibit  that the convergence rate  \eqref{C_GDA} is exact for one iteration by using problem \eqref{bP}; see Proposition \ref{Pro.2}.

Theorem \ref{T_GDA} provides some new information concerning Algorithm \ref{GDA}. Firstly, Theorem \ref{T_GDA} improves the known convergence factor in the literature; see our discussion in Introduction. In addition, it investigates the  convergence rate for a step length in a larger interval. Secondly, it does not assume the second order continuous  differentiability of $F$, which is commonly used for deriving a local convergence rate; see \cite{liang2019interaction, mescheder2017numerics, zhang2022near}.  Finally, the given convergence rate incorporates three parameter $\mu=\min\{\mu_x, \mu_y\}$, $L=\max\{L_x, L_y\}$ and $L_{xy}$, which is more informative in comparison  with the results in the literature mostly given in terms of $\mu=\min\{\mu_x, \mu_y\}$ and $L=\max\{L_x, L_y, L_{xy}\}$; see \cite{mokhtari2020unified, zhang2022near, zhang2022lower} and references therein. Even though if one considers $L=\max\{L_x, L_y, L_{xy}\}$ and $\mu=\min\{\mu_x, \mu_y\}$, convergence rate \eqref{C_GDA} dominates \eqref{V.GDA}. This follows from that  for $t\in \left(0, \tfrac{\mu}{2L^2}\right)$, one has
{\small{
\begin{align*}
&\left(1+4L^2t^2-2\mu t\right)-\left(1+\tfrac{1}{2}\left(3L^2+\mu^2\right)t^2-(L+\mu)t +\tfrac{1}{2}(L-\mu)t\sqrt{(Lt + \mu t - 2)^2 + 4L^2t^2}\right)\\
& \geq (2L^2+L\mu-\mu^2)t^2\geq 2L^2t^2,
\end{align*}
}}
where the first inequality results from $\sqrt{(Lt + \mu t - 2)^2 + 4L^2t^2}\leq (2-Lt-\mu t)+2Lt$. In addition, in this case, the step length can take value in a larger interval as $\left(0, \tfrac{\mu}{2L^2}\right)\subseteq  \left(0, \tfrac{2\mu}{L(L+\mu)}\right)$.
Moreover, Conjecture \ref{Conj1} discusses the convergence rate in terms of  $L_x, L_y, L_{xy}, \mu_x, \mu_y$.

The next proposition gives the optimal step length with respect to the worst case convergence rate.
\begin{proposition}\label{Pro.1}
Let $F\in\mathcal{F}(L_x, L_y, L_{xy}, \mu_x, \mu_y)$. If $L=\max\{L_x, L_y\}$ and $\mu=\min\{\mu_x, \mu_y\}>0$,   then the optimal step length for Algorithm \ref{GDA} with respect to bound \eqref{C_GDA} is
\begin{align}\label{Opt_t}
t^\star=\tfrac{2\left((L+\mu)\sqrt{L_{xy}^2+L\mu}+L_{xy}(\mu-L)\right)}
{\left(4L_{xy}^2+(L+\mu)^2\right)\sqrt{L_{xy}^2+L\mu}}.
\end{align}
Moreover, the convergence rate with respect to $t^\star$ is
\begin{align}\label{Opt_al}
\alpha^\star=\tfrac{8 L_{xy} \left(L^2-\mu^2\right) \sqrt{L \mu+L_{xy}^2}+\left(L^2-\mu^2\right)^2+16 L_{xy}^2 \left(L \mu+L_{xy}^2\right)}{\left((L+\mu)^2+4 L_{xy}^2\right)^2}.
\end{align}
\end{proposition}
\begin{proof}
Let $\alpha:\left[0, \tfrac{2\mu}{\mu L+L_{xy}^2}\right] \to\mathbb{R}$ given by
$$
\alpha(t)=1+\tfrac{1}{2}\left(L^2+\mu^2+2L_{xy}^2\right)t^2-(L+\mu)t +\tfrac{1}{2}(L-\mu)t\sqrt{(Lt + \mu t - 2)^2 + 4L_{xy}^2t^2}.
$$
By Lemma \ref{Lemma1}, $\alpha$ is a strictly convex function on its domain. By doing some algebra, one can verify that $\alpha^\prime (t^\star)=0$, which implies that $t^\star$ is the minimum.
\end{proof}

If $L_{xy}=0$, problem \eqref{P} reduces to a separable optimization problem. Indeed, the variables $x$ and $y$ are independent. Under this assumption, the optimal step length given by Proposition \ref{Pro.1} is $t^\star=\tfrac{2}{L+\mu}$, which is the well-known optimal step length  for the optimization problem
\begin{align*}
\min_{x\in \mathbb{R}^n} f(x),
\end{align*}
where $f\in\mathcal{F}_{\mu, L}$; see \cite[Theorem 2.1.15]{Nesterov}. Moreover, the convergence rate corresponding to $t^\star$ is $\alpha^\star=\left(\tfrac{L-\mu}{L+\mu}\right)^2$. By some algebra, one can show that under the assumptions of Proposition \eqref{Pro.1}, Algorithm \ref{GDA} has a complexity of
$\mathcal{O}\left(\left(\tfrac{L}{\mu}+\tfrac{L_{xy}^2}{\mu^2}\right)\ln \left(\tfrac{1}{\epsilon} \right)\right)$. Note that the lower iteration complexity bound for first order methods with $L=\max\{L_x, L_y\}$ and $\mu=\min\{\mu_x, \mu_y\}$ is $\Omega\left(\sqrt{\tfrac{L}{\mu}+\tfrac{L_{xy}^2}{\mu^2}}\ln \left(\tfrac{1}{\epsilon} \right)\right)$; see \cite{zhang2022lower}.

As mentioned earlier, we calculated the convergence rate by using problem \eqref{bP}. The next proposition states that the bound \eqref{C_GDA} is tight for some class of bilinear saddle point problems.

\begin{proposition}\label{Pro.2}
Let $F\in\mathcal{F}(L_x, L_y, L_{xy}, \mu_x, \mu_y)$. Suppose that $L_x=L_y$ and $\min\{\mu_x, \mu_y\}>0$.   If $t\in\left(0, \tfrac{2\mu}{\mu L+L_{xy}^2}\right)$, then convergence rate \eqref{C_GDA} is exact for one iteration.
\end{proposition}

\begin{proof}
To establish the proposition, it suffices to introduce a problem for which Algorithm \ref{GDA} generates $(x^2, y^2)$ with respect to the initial point $(x^1, y^1)$ such that
$$
\|x^2-x^\star\|^2+\|y^2-y^\star\|^2=
\alpha\left(\|x^1-x^\star\|^2+\|y^1-y^\star\|^2\right),
$$
where $\alpha$ is the convergence rate factor given in Theorem \ref{T_GDA}. Consider problem \eqref{bP}. Due to the symmetry of Algorithm \ref{GDA} and the class of problems, we may assume  $\mu_x\geq \mu_y$. Moreover, without loss of generality, we can take  $L_{xy}=1$; see our discussion in the proof of Theorem \ref{T_GDA}. Suppose $L=L_x$, $\mu=\mu_y$ and $\beta=\sqrt{(L t+\mu t-2)^2+4 t^2}$. One can verify that Algorithm \ref{GDA}  with the initial point
\begin{align*}
& x^1_1=0, && x_2^1=\sqrt{\tfrac{2-t\left(L+\mu\right)+\beta}{2\beta}},\\
& y_1^1=-t\sqrt{\tfrac{2}{\beta\left(2-t\left(L+\mu\right)+\beta\right)}},
&& y^1_2=0.
\end{align*}
generates $(x^2, y^2)$ with the desired equality.
\end{proof}

One may wonder why we stress on one iteration in Proposition \ref{Pro.2}. Based on our numerical results if $L_{xy}>0$, under the setting of Theorem \ref{T_GDA}, we observed that
$$
\|x^k-x^\star\|^2+\|y^k-y^\star\|^2<
\alpha^{k-1}\left(\|x^1-x^\star\|^2+\|y^1-y^\star\|^2\right), \ \ \ k\geq 3,
$$
for some $t\in\left(0, \tfrac{2\mu}{\mu L+L_{xy}^2}\right)$. The reason may be related to the fact that the  vector field
$\begin{pmatrix}
\nabla_x F(x,y) & -\nabla_y F(x,y)
\end{pmatrix}^T$
 is not conservative.

It may be of interest whether  inequality \eqref{C_GDA} may hold without strong convexity. By removing strong convexity, the solution set may not be singleton. Hence, we investigate distance to the solution set, that is, if there exists $0\leq \alpha<1$ with
$$
d_{S^\star}^2((x^2, y^2))\leq  \alpha d_{S^\star}^2((x^1, y^1)).
$$
The next proposition says in general the answer is negative. Indeed, it gives an example with $\min\{\mu_x, \mu_y\}=0$ and a unique saddle point  for which
$$
\|x^2-x^\star\|^2+\|y^2-y^\star\|^2\geq
\alpha\left(\|x^1-x^\star\|^2+\|y^1-y^\star\|^2\right),
$$
for some $\alpha\geq 1$, no matter how close $(x^1, y^1)$ is to the unique saddle point and which positive step length $t$ is taken. In the next proposition, we may assume without loss of generality $\mu_x=0$ and make an example analogous to that given in Proposition \ref{Pro.2}.

\begin{proposition}
Let $L, L_{xy},\mu_y, t, r>0$  be given. Then there exists a function $F\in\mathcal{F}(L, L, L_{xy}, 0, \mu_y)$ with the unique saddle point $(x^\star, y^\star)$ and $(x^1, y^1)$ such that, for $(x^2, y^2)$ generated by Algorithm \ref{GDA}, we have
$$
\|x^2-x^\star\|^2+\|y^2-y^\star\|^2\geq
\alpha\left(\|x^1-x^\star\|^2+\|y^1-y^\star\|^2\right),
$$
for some $\alpha\geq 1$ and $\|x^1-x^\star\|^2+\|y^1-y^\star\|^2=r^2$.
\end{proposition}
\begin{proof}
As discussed before,  we may assume  $L_{xy}=1$. Consider the bilinear saddle point problem,
\begin{align*}
\min_{x\in \mathbb{R}^2}\max_{y\in \mathbb{R}^2} F(x, y)=\tfrac{1}{2}x^T\begin{pmatrix}
L & 0\\
0& 0
\end{pmatrix}x+x^T\begin{pmatrix}
    0 & 1 \\
    1 & 0
  \end{pmatrix}y-\tfrac{1}{2}y^T\begin{pmatrix}
L & 0\\
0&\mu_y
\end{pmatrix}y.
\end{align*}
It is seen that $F\in\mathcal{F}(L, L, L_{xy}, 0, \mu_y)$ and the unique saddle point is $(x^\star, y^\star)=(0, 0)$. Suppose that
\begin{align*}
& x^1_1=0, && x_2^1=r\sqrt{\tfrac{2-tL+\beta}{2\beta}},\\
& y_1^1=-rt\sqrt{\tfrac{2}{\beta\left(2-tL+\beta\right)}},
&& y^1_2=0,
\end{align*}
where $\beta=\sqrt{(L t-2)^2+4 t^2}$. One can verify Algorithm \ref{GDA} generates $(x^2, y^2)$ with
$$
\|x^2-x^\star\|^2+\|y^2-y^\star\|^2\geq
\alpha\left(\|x^1-x^\star\|^2+\|y^1-y^\star\|^2\right)=\alpha r^2,
$$
where
$\alpha=1+\tfrac{1}{2}\left(L^2+2\right)t^2-Lt +\tfrac{1}{2}Lt\sqrt{(Lt- 2)^2 + 4t^2}$. By Proposition \ref{Pro.1}, one can infer that $\alpha\geq 1$.
\end{proof}

By Proposition \ref{Pro.2}, one can infer that the convergence rate factor for bilinear saddle point problems may not be improved for one iteration since the given example is a bilinear saddle point problem. Furthermore, the given convergence rate factor is tight whether $L_x=L_y$. As discussed in \cite{wang2020improved},  the function $H(x, y)=F\left(\sqrt[4]{\tfrac{L_y}{L_x}}x, \sqrt[4]{\tfrac{L_x}{L_y}}y\right)$ shares the same smoothness constants with respect to $x$ and $y$, that is, $\nabla_x H(\cdot ,y)$ and $\nabla_y H(x ,\cdot)$ are Lipschitz continuous with the same modulus $\sqrt{L_x L_y}$. However, the gradient methods are not invariant under scaling; see \cite[Chapter 9]{boyd2004convex}. Hence, we may lose the generality of our discussion by assuming this condition.

Based on our numerical results and analysis of problem \eqref{bP}, we propose the following conjecture concerning the convergence rate of Algorithm \ref{GDA} in terms of  $L_x, L_y, L_{xy}, \mu_x, \mu_y$. Due to the symmetry of Algorithm \ref{GDA}, we may assume that $L_x\geq L_y$. Moreover, Proposition \ref{Pro.2} implies that bound \eqref{C_GDA} is tight when $\mu_y\leq\mu_x$. Hence, we need only consider $\mu_y>\mu_x$.

\begin{conjecture}\label{Conj1}
Let $F\in\mathcal{F}(L_x, L_y, L_{xy}, \mu_x, \mu_y)$. Suppose that $\mu_y>\mu_x>0$, $\max\{L_x, L_y\}=L_x$  and
{\footnotesize{
\begin{align*}
  & c=\tfrac{1}{2}(L_y^2+ \mu_x^2)t-(L_y +\mu_x) +\tfrac{1}{2}(L_y - \mu_x)\sqrt{(L_yt+ \mu_xt - 2)^2 + 4L_{xy}^2t^2},\\
  & \bar\mu=\tfrac{c + 2L_x-L_x^2t+L_xL_{xy}^2t^2 - (c+L_x(2- L_xt))\sqrt{1+t(c + tL_{xy}^2)}}{tL_{xy}(c+ tL_{xy}^2+ L_x(2- L_xt))},\\
&\alpha(\mu, L, L_{xy}, t)=1+\tfrac{1}{2}\left(L^2+\mu^2+2L_{xy}^2\right)t^2-(L+\mu)t +\tfrac{1}{2}(L-\mu)t\sqrt{(Lt + \mu t - 2)^2 + 4L_{xy}^2t^2}.
\end{align*}
}}
\begin{enumerate}[a)]
  \item 
  Assume that $\mu_x\mu_y(L_x-L_y)\geq L_{xy}^2\left(\mu_y-\mu_x\right)$ and $t\in\left(0,  \tfrac{2\mu_y}{L_x\mu_y+L_{xy}^2}\right)$.
  \begin{enumerate}[i)]
    \item
    If $\mu_y\leq \bar\mu$, then
    $$
    \|x^2-x^\star\|^2+\|y^2-y^\star\|^2\leq
\alpha(\mu_y, L_x, L_{xy}, t)\left(\|x^1-x^\star\|^2+\|y^1-y^\star\|^2\right).
    $$
    \item
    If $\mu_y\geq \bar\mu$, then
    $$
    \|x^2-x^\star\|^2+\|y^2-y^\star\|^2\leq
\alpha(\mu_x, L_y, L_{xy}, t)\left(\|x^1-x^\star\|^2+\|y^1-y^\star\|^2\right).
    $$

  \end{enumerate}
  \item 
  Assume that $\mu_x\mu_y(L_x-L_y)\leq L_{xy}^2\left(\mu_y-\mu_x\right)$ and $t\in\left(0,  \tfrac{2\mu_x}{L_y\mu_x+L_{xy}^2}\right)$.
\begin{enumerate}[i)]
    \item
    If $\mu_y\leq \bar\mu$, then
    $$
    \|x^2-x^\star\|^2+\|y^2-y^\star\|^2\leq
\alpha(\mu_y, L_x, L_{xy}, t)\left(\|x^1-x^\star\|^2+\|y^1-y^\star\|^2\right).
    $$
    \item
    If $\mu_y\geq \bar\mu$, then
    $$
    \|x^2-x^\star\|^2+\|y^2-y^\star\|^2\leq
\alpha(\mu_x, L_y, L_{xy}, t)\left(\|x^1-x^\star\|^2+\|y^1-y^\star\|^2\right).
    $$

  \end{enumerate}
\end{enumerate}
\end{conjecture}
Although we have extensive numerical evidence supporting Conjecture \ref{Conj1}, we have been unable to prove it.
\section{Linear convergence  without strong convexity}\label{Sec.2}
In this section, we study the linear convergence of Algorithm \ref{GDA} without assuming strong convexity. Indeed, we suppose that $F\in\mathcal{F}\mathcal(L_x, L_y, L_{xy}, 0, 0)$ and we propose some necessary and sufficient conditions for the linear convergence. This subject has received some attention in recent years and some sufficient conditions have been proposed in \cite{du2019linear, zhang2022near} under which  Algorithm \ref{GDA} enjoys local linear convergence rate or it is linearly convergent for bilinear saddle point problems. This topic has been investigated extensively in the context of optimization. The interested reader can refer to \cite{abbaszadehpeivasti2022conditions, bolte2017error, luo1993error, necoara2019linear} and references therein. In this study,  we extend the quadratic gradient growth property introduced in \cite{luo1993error} for saddle point problems.

Recall that we denote the nonempty solution set of problem \eqref{P} by $S^\star$. As we do not assume the strong convexity (concavity), $S^\star$ may not be singleton. Note that $S^\star$ is a closed convex set under our assumptions. Recall that $\Pi_{S^\star} \left((x, y)\right)$ denotes the projection of $(x, y)$ onto $S^\star$.

\begin{definition}
Let $\mu_F>0$. A function $F$ has a quadratic gradient growth  if for any $x\in\mathbb{R}^n$ and $y\in\mathbb{R}^m$,
\begin{align}\label{Qq_r}
\langle \nabla_x F(x,y), x-x^\star\rangle-\langle \nabla_y F(x,y), y-y^\star\rangle \geq \mu_F d_{S^\star}^2((x,y)),
\end{align}
where $(x^\star, y^\star)= \Pi_{S^\star} \left((x, y)\right)$.
\end{definition}

Note that if we set $y=y^\star$ in \eqref{Qq_r}, we have
$$
\langle \nabla_x F(x,y^\star), x-x^\star\rangle\geq  \mu_F \|x-x^\star\|^2.
$$
Hence, $L_x$-smoothness implies that $\mu_F\leq L_x$. Consequently,  due to the symmetry, we have $\mu_F\leq \min\{L_x, L_y\}$. The next proposition states that the quadratic gradient growth condition is weaker than the strong convexity-strong concavity. Indeed, the strong convexity-strong concavity implies the quadratic gradient growth property.

\begin{proposition}\label{Por.3}
Let $F\in\mathcal{F}\mathcal(L_x, L_y, L_{xy}, \mu_x, \mu_y)$. If $\min\{\mu_x, \mu_y\}>0$, then $F$ has a quadratic gradient growth with $\mu_F=\min\{\mu_x, \mu_y\}$.
\end{proposition}
\begin{proof}
Under the assumptions, problem \eqref{P} has a unique solution $(x^\star, y^\star)$ and $\nabla_x F(x^\star, y^\star)=0$ and  $\nabla_y F(x^\star, y^\star)=0$. Let $\mu=\min\{\mu_x, \mu_y\}$ and $L=\max\{L_x, L_y\}$. Suppose that  $(x, y)\in\mathbb{R}^n\times\mathbb{R}^m$. By Theorem \ref{T1},  we have
{\small{
\begin{align*}
0\leq & \bigg(F(x^\star,y)-F(x,y)+\left\langle \nabla_x F(x,y),x-x^\star\right\rangle-\tfrac{L}{2\left(L-\mu\right)}\Big(\tfrac{1}{L}\left\|\nabla_x F(x^\star,y)-\nabla_x F(x,y)\right\|^2+\\
  &\mu\left\|x-x^\star\right\|^2-\tfrac{2\mu}{L}\left\langle \nabla_x F(x,y)-\nabla_x F(x^\star,y),x-x^\star\right\rangle\Big)\bigg)+\bigg(F(x,y^\star)-F(x^\star,y^\star)-\\
  &\tfrac{L}{2\left(L-\mu\right)}\Big(\tfrac{1}{L}\left\|\nabla_x F(x,y^\star)\right\|^2+\mu\left\|x-x^\star\right\|^2-\tfrac{2\mu}{L}\left\langle \nabla_x F(x,y^\star),x-x^\star\right\rangle\Big)\bigg)+\bigg(F(x,y)-\\
  &F(x,y^\star)-\left\langle \nabla_y F(x,y),y-y^\star\right\rangle-\tfrac{L}{2\left(L-\mu\right)}\Big(\tfrac{1}{L}\left\|\nabla_y F(x,y^\star)-\nabla_y F(x,y)\right\|^2+\mu\left\|y-y^\star\right\|^2-\\
  &\tfrac{2\mu}{L}\left\langle\nabla_y F(x,y^\star)-\nabla_y F(x,y),y-y^\star\right\rangle\Big)\bigg)+\bigg(F(x^\star,y^\star)-F(x^\star,y)-\tfrac{L}{2\left(L-\mu\right)}\\
  &\Big(\tfrac{1}{L}\left\|\nabla_y F(x^\star,y)\right\|^2+\mu\left\|y-y^\star\right\|^2-\tfrac{2\mu}{L}\left\langle \nabla_y F(x^\star,y),y^\star-y\right\rangle\Big)\bigg)\\
  = & \tfrac{-\mu^2}{L-\mu}\left\|\left(x-x^\star\right)-\tfrac{1}{2\mu}\left(\nabla_x F(x,y)+\nabla_x F(x,y^\star)-\nabla_x F(x^\star,y)\right)\right\|^2-\\
  &\tfrac{1}{4\left(L-\mu\right)}\left\|\nabla_x F(x,y)-\nabla_x F(x,y^\star)-\nabla_x F(x^\star,y)\right\|^2-\\
  & \tfrac{\mu^2}{L-\mu}\left\|\left(y-y^\star\right)+\tfrac{1}{2\mu}\left(\nabla_y F(x,y)-\nabla_y F(x,y^\star)+\nabla_y F(x^\star,y)\right)\right\|^2-\\
  &\tfrac{1}{4\left(L-\mu\right)}\left\|\nabla_y F(x,y)-\nabla_y F(x,y^\star)-\nabla_y F(x^\star,y)\right\|^2-\\
  & \mu\left(\left\|x-x^\star\right\|^2+\left\|y-y^\star\right\|^2\right)+\left\langle \nabla_x F(x,y),x-x^\star\right\rangle-\left\langle \nabla_y F(x,y),y-y^\star\right\rangle.
\end{align*}
}}
Hence,
$$
\mu\left(\left\|x-x^\star\right\|^2+\left\|y-y^\star\right\|^2\right)\leq
\left\langle \nabla_x F(x,y),x-x^\star\right\rangle-\left\langle \nabla_y F(x,y),y-y^\star\right\rangle,
$$
and the proof is complete.
\end{proof}

Note that the converse of Proposition \ref{Por.3} does not hold necessarily. Consider the following saddle point problem
\begin{align}\label{Exe.1}
\min_{x\in \mathbb{R}}\max_{y\in \mathbb{R}} F(x,y):=f(x+y)-2y^2,
\end{align}
where
$$
f(s)=\begin{cases}
  0 & |s|\leq 1 \\
  (s-1)^2 & s>1 \\
  (s+1)^2 & s<-1.
\end{cases}
$$
It is seen that $F$ is not strongly convex-strongly concave and the solution set of problem \eqref{Exe.1} is $\{(x, 0): |x|\leq 1\}$. By doing some algebra, one can check that $F$ has a quadratic gradient growth with $\mu_F=1$ while it is not strongly convex with respect to the first component. For the case that  $F(\cdot, y)$ is  neither strongly convex nor is $F(x, \cdot)$ strongly concave, one may consider uncoupled problem
$\min_{x\in \mathbb{R}}\max_{y\in \mathbb{R}} f(x)-f(y)$.

 In what follows, by using performance estimation, we establish that Algorithm \ref{GDA} enjoys the linear convergence whether $F\in\mathcal{F}\mathcal(L_x, L_y, L_{xy}, 0, 0)$ has a quadratic gradient growth. Without loss of generality, we may assume that
 $(0, 0)=\Pi_{S^\star} \left((x^1, y^1)\right)$. To establish the linear convergence, it suffices to show that
 $$
d_{S^\star}^2((x^2, y^2))\leq \|x^2\|^2+\|y^2\|^2\leq \alpha d_{S^\star}^2((x^1, y^1)),
 $$
 for some $\alpha\in[0, 1)$. Similarly to Section \ref{Sec.1}, we formulate the following optimization problem
\begin{align}\label{P_QQ}
\nonumber   \max & \ \frac{\|x^1-tG_x^{1,1}\|^2+\|y^1+tG_y^{1,1}\|^2}{\|x^1\|^2+\|y^1\|^2}\\
 \nonumber \st   &  \  \{(x^1; G_x^{1,k}; F^{1,k}),(x^1-tG_x^{1,1}; G_x^{2,k}; F^{2,k}),(0;  G_x^{\star,k}; F^{\star,k})\} \textrm{ satisty \eqref{interp.1} for}\\
  \nonumber    &\ \ \ \ \  \textrm{ $k\in\{1,2, \star\}$ w.r.t.}\ \mu_x=0, L_x \\
   \nonumber    &  \  \{(y^1; G_y^{k,1}; F^{k,1}),(y^1+tG_y^{1,1}; G_y^{k,2}; F^{k,2}),(0;  G_y^{k,\star}; F^{k,\star})\}  \ \textrm{satisty \eqref{interp.2} for}\\
     &\ \ \ \ \ \textrm{ $k\in\{1,2, \star, *\}$ w.r.t.}\ \mu_y=0, L_y \\
  \nonumber & \|G_x^{k,i}-G_x^{k,j}\|\leq L_{xy}\|y^i-y^j\|, \ \ \ i,j,k\in\{1, 2, \star\}\\
   \nonumber & \|G_y^{i,k}-G_y^{j,k}\|\leq L_{xy}\|x^i-x^j\|, \ \ \ i,j,k\in\{1, 2, \star\}\\
 \nonumber & \mu_F \left( \|x^1\|^2+\|y^1\|^2\right)\leq \langle G_x^{1,1}, x^1\rangle-\langle G_y^{1,1}, y^1\rangle,\\
\nonumber &   G_x^{\star, \star}=0, \ G_y^{\star, \star}=0.
\end{align}

Note that in the formulation \eqref{P_QQ}, we only use a subset of constraints for the performance estimation. In the next theorem, we prove the linear convergence of Algorithm \ref{GDA} when $F$ has a quadratic gradient growth.

\begin{theorem}\label{T_GDA_P}
Let $F\in\mathcal{F}(L_x, L_y, L_{xy}, 0, 0)$ and $L=\max\{L_x, L_y\}$. Assume that $F$  has a quadratic gradient growth with $\mu_F>0$. If  $t\in\left(0, \tfrac{2 \mu_F}{L \mu_F+2L_{xy} \sqrt{\mu_F (L-\mu_F)}+L_{xy}^2}\right)$, then Algorithm \ref{GDA} generates $(x^2, y^2)$ such that
\begin{align}\label{C_GDA_P}
d_{S^\star}^2((x^2, y^2))\leq  \alpha d_{S^\star}^2((x^1, y^1)),
\end{align}
where
$$
\alpha=t \left(2 tL_{xy} \sqrt{\mu_F (L-\mu_F)}+\mu_F (L t-2)+tL_{xy}^2\right)+1.
$$
\end{theorem}
\begin{proof}
The argument is similar to that of Theorem \ref{T_GDA}. It is seen that for any step length $t$ in the given interval, $\alpha\in [0, 1)$. We may assume without loss of generality $L_{xy}=1$. By the assumptions,
$F(\cdot, y)\in\mathcal{F}_{0,L}(\mathbb{R}^n)$ and $F(x, \cdot)\in\mathcal{F}_{0,L}(\mathbb{R}^m)$ for any fixed $x, y$. Suppose that
 \begin{align*}
  &\bar\alpha=t \left(2 t \sqrt{\mu_F (L-\mu_F)}+\mu_F (L t-2)+t\right)+1, &&\beta=t^2\left(\mu_F\sqrt{L-\mu_F}+\sqrt{\mu_F}\right),\\
  &\gamma_1=t^2 \left(\tfrac{\mu_F}{\sqrt{\mu_F (L-\mu_F)}}+\mu_F\right), &&\gamma_2=\tfrac{t^2 \left(\mu_F (L-\mu_F)+\sqrt{\mu_F (L-\mu_F)}\right)}{\mu_F},\\
   &\gamma_3=-\tfrac{t^2 \left(\mu_F (L+\mu_F)+\sqrt{\mu_F (L-\mu_F)}\right)}{\mu_F}+\tfrac{\beta}{\sqrt{L-\mu_F}}+2 t, && \gamma_4=\tfrac{1}{2} t^2 \left(\sqrt{\mu_F (L-\mu_F)}+1\right).
\end{align*}
One may readily verify that $\gamma_1,\gamma_2,\gamma_3,\gamma_4\geq 0$. By doing some algebra, one can show that
{\small{
\begin{align*}
& \left\|x^1-tG_x^{1,1}\right\|^2+\left\|y^1+tG_y^{1,1}\right\|^2-
\bar\alpha\left(\left\|x^1\right\|^2+\left\|y^1\right\|^2\right)+\gamma_1\bigg(F^{1,1}-F^{\star,1}-\left\langle G_x^{\star,1},x^1\right\rangle-\\
  &\tfrac{1}{2L}\left\|G_x^{1,1}-G_x^{\star,1}\right\|^2\bigg)+\gamma_2\bigg(F^{\star,1}-F^{1,1}+\left\langle G_x^{1,1},x^1\right\rangle-\tfrac{1}{2L}\left\|G_x^{\star,1}-G_x^{1,1}\right\|^2\bigg)+\gamma_2\bigg(F^{1,\star}-\\
  & F^{\star,\star}-\tfrac{1}{2L}\left\|G_x^{1,\star}\right\|^2\bigg)+\gamma_1\bigg(F^{\star,\star}-F^{1,\star}+\left\langle G_x^{1,\star},x^1\right\rangle-\tfrac{1}{2L}\left\|G_x^{1,\star}\right\|^2\bigg)+\gamma_1\bigg(F^{1,\star}-F^{1,1}+\\
  &\left\langle G_y^{1,\star},y^1\right\rangle-\tfrac{1}{2L}\left\|G_y^{1,1}-G_y^{1,\star}\right\|^2\bigg)+\gamma_2\bigg(F^{1,1}-F^{1,\star}-\left\langle G_y^{1,1},y^1\right\rangle-\tfrac{1}{2L}\left\|G_y^{1,\star}-G_y^{1,1}\right\|^2\bigg)+\\
  &\gamma_2\bigg(-F^{\star,1}+F^{\star,\star}-\tfrac{1}{2L}\left\|G_y^{\star,1}\right\|^2\bigg)+\gamma_1\bigg(-F^{\star,\star}+F^{\star,1}+\left\langle G_y^{\star,1},-y^1\right\rangle-\tfrac{1}{2L}\left\|G_y^{\star,1}\right\|^2\bigg)+\\
  & \gamma_3\bigg(\left\langle G_x^{1,1},x^1\right\rangle-\left\langle G_y^{1,1},y^1\right\rangle-\mu_F\Big(\left\|x^1\right\|^2+\left\|y^1\right\|^2\Big)\bigg)+\gamma_4\left(\left\|x^1\right\|^2-\left\|G_y^{1,1}-G_y^{\star,1}\right\|^2\right)+\\
  & \gamma_4\left(\left\|x^1\right\|^2-\left\|G_y^{1,\star}\right\|^2\right)+\gamma_4\left(\left\|y^1\right\|^2-\left\|G_x^{1,1}-G_x^{1,\star}\right\|^2\right)+\gamma_4\left(\left\|y^1\right\|^2-\left\|G_x^{\star,1}\right\|^2\right)\\
 = & -\zeta_1\left\|x^1+\zeta_2 G_x^{1,1}-\zeta_3(G_x^{1,\star}-G_x^{\star,1})\right\|^2-\zeta_4\left\|G_x^{1,1}-G_x^{1,\star}-G_x^{\star,1}\right\|^2-\\
  & \zeta_1\left\|y^1-\zeta_2 G_y^{1,1}-\zeta_3(G_y^{1,\star}-G_y^{\star,1})\right\|^2-\zeta_4\left\|G_y^{1,1}-G_y^{\star,1}-G_y^{1,\star}\right\|^2\leq 0,
\end{align*}
}}
where the multipliers $\zeta_1, \zeta_2, \zeta_3, \zeta_4$ are given as follows
\begin{align*}
  & \zeta_1=\mu_F \left(\tfrac{\beta}{\sqrt{L-\mu_F}}-\mu_F t^2\right),\
  \zeta_2=\tfrac{\beta}{2 \mu_F\sqrt{\mu_F}t^2}-\tfrac{1}{\mu_F}, \
   \zeta_3=\tfrac{t^2 \left(\sqrt{\mu_F (L-\mu_F)}+1\right)}{2 \mu_F t^2}, \\
  & \zeta_4=\tfrac{1}{4} \left(\tfrac{2 t^2 (\mu_F (L-\mu_F)+1)}{\sqrt{\mu_F (L-\mu_F)}}-\tfrac{\left(2 \mu_F t^2 (\mu_F-L)+\beta\sqrt{L-\mu_F}\right)^2}{\mu_F (L-\mu_F)t^2 \sqrt{\mu_F\left(L-\mu_F\right)}}\right).
\end{align*}
One can show by some algebra that $\zeta_1, \zeta_4\geq 0$. Hence, for any feasible solution of problem \eqref{P_QQ}, we have
\begin{align*}
   \frac{\left\|x^1-tG_x^{1,1}\right\|^2+\left\|y^1+tG_y^{1,1}\right\|^2}{\left\|x^1\right\|^2+\left\|y^1\right\|^2}\leq \bar\alpha,
\end{align*}
and the proof is complete.
\end{proof}

We obtained the  linear convergence by using quadratic gradient growth in Theorem \ref{T_GDA_P}. The next theorem states that quadratic gradient growth property is also a sufficient condition for the linear convergence.

\begin{theorem}
If Algorithm \ref{GDA} is linearly convergent for any initial point, then $F$ has a quadratic gradient growth for some $\mu_F>0$.
\end{theorem}
\begin{proof}
  Let $(x^1, y^1)\in\mathbb{R}^n\times\mathbb{R}^m$ and $(x^2, y^2)$ be generated by Algorithm \ref{GDA}.  Suppose that $(x^\star, y^\star)=\Pi_{S^\star} \left((x^2, y^2\right))$.
  As Algorithm \ref{GDA} is linearly convergent, there exist $\alpha\in [0,1)$ with
  \begin{align}\label{ccc}
  d_{S^\star}^2((x^2, y^2))\leq \alpha  d_{S^\star}^2((x^1, y^1))\leq \alpha\left(\|x^1-x^\star\|^2+\|y^1-y^\star\|^2\right).
  \end{align}
 By setting $x^2=x^1-t\nabla_x F(x^1, y^1)$ and $y^2=y^1+t\nabla_y F(x^1, y^1)$ in inequality \eqref{ccc}, we get
 $$
 \tfrac{1-\alpha}{2t}\left( \|x^1-x^\star\|^2+\|y^1-y^\star\|^2\right)\leq \left\langle \nabla_x F(x^1,y^1),x^1-x^\star\right\rangle-\left\langle \nabla_y F(x^1,y^1),y^1-y^\star\right\rangle,
 $$
 which implies that
 $$
 \mu_F d_{S^\star}^2(x^1, y^1)\leq \left\langle \nabla_x F(x^1,y^1),x^1-x^\star\right\rangle-\left\langle \nabla_y F(x^1,y^1),y^1-y^\star\right\rangle,
 $$
 for $\mu_F=\tfrac{1-\alpha}{2t}$ and the proof is complete.
\end{proof}

\section*{Concluding remarks}
In this study, we provided a new convergence rate for the gradient descent-ascent   method for saddle point problems. Furthermore, we gave some necessary and sufficient conditions for the linear convergence without strong convexity. We employed performance estimation method for proving the results.  For future work, it would be interesting to consider the case where the variables $x$ and $y$ in the saddle point problem are constrained to lie in given, compact convex sets, since many saddle point problems fall in this category. In this case, one could use the performance estimation framework to analyze other methods, e.g. proximal type algorithms.

%
%

\bibliographystyle{siamplain}      
\bibliography{references}
\end{document}